\documentclass[11pt]{amsart}
\usepackage{amsfonts}
\usepackage{graphicx,color}
\usepackage{amsthm}
\usepackage{amssymb,amsmath}
\usepackage{url}

\title[Rationality of the inner products of spherical $s$-distance $t$-designs for $t \geq 2s-2$, $s \geq 3$]{\bf Rationality of the inner products of spherical $s$-distance $t$-designs for $t \geq 2s-2$, $s \geq 3$}

\author[P. Boyvalenkov]{Peter Boyvalenkov$^\dagger$}
\address{Institute of Mathematics and Informatics, Bulgarian Academy of Sciences,
8 G Bonchev Str.,
1113  Sofia, Bulgaria}
\email{peter@math.bas.bg}
\thanks{\noindent $^\dagger$ The research of the first author was supported, in part, by Bulgarian NSF under project KP-06-N32/2-2019. }
\author[H. Nozaki]{Hiroshi Nozaki$^\ddag$}
\address{Aichi University of Education, Department of Mathematics Education
1 Hirosawa, Igaya-cho, Kariya-city, Aichi, 448-8542, Japan}
\email{hnozaki@auecc.aichi-edu.ac.jp}
\thanks{$^\ddag$ The second author is supported by JSPS KAKENHI Grant Numbers 18K03396, 19K03445 and 20K03527.}
\author[N. Safaei]{Navid Safaei$^{*}$}
\address{Research Institute of  Policy Making, Sharif University of Technology,
Tehran, Iran}
\email{navid\_safaei@gsme.sharif.edu}

\date{\today}

\newtheorem{theorem}{Theorem}[section]
\newtheorem{lemma}[theorem]{Lemma}
\newtheorem{corollary}[theorem]{Corollary}

\theoremstyle{definition}
\newtheorem{definition}[theorem]{Definition}

\newtheorem{remark}[theorem]{Remark}

\begin{document}
\maketitle

\begin{abstract}
We prove that the inner products of spherical $s$-distance $t$-designs with $t \geq 2s-2$ (Delsarte codes) and $s \geq 3$  are rational with the only exception being the icosahedron. In other formulations, we prove that all sharp configurations have rational inner products and all spherical codes which attain the Levenshtein bound, have rational inner products, except for the icosahedron. 
\end{abstract}



\section{Introduction}

Let $\mathbb{S}^{n-1}$ be the unit sphere in $\mathbb{R}^n$. A finite set $C \subset \mathbb{S}^{n-1}$ is called a spherical code. A special class of spherical codes, called {\it spherical designs}, was introduced by Delsarte, Goethals and Seidel in 1977 in a seminal paper \cite{DGS}.

\begin{definition} \label{def-designs} A spherical code $C \subset \mathbb{S}^{n-1}$ is called a spherical $t$-design if the quadrature 
formula
\[  \frac{1}{\mu(\mathbb{S}^{n-1})} \int_{\mathbb{S}^{n-1}} f(x) d\mu(x) =
                  \frac{1}{|C|} \sum_{x \in C} f(x) \]
is exact for all polynomials $f(x)= f(x_1,x_2,\ldots,x_n)$ of degree at most $t$. 
\end{definition}

For a spherical code $C$ we consider the set $A(C):=\{ \langle x,y \rangle : x,y \in C, x \neq y\}$ and denote by $s:=|A(C)|$ the number of distinct inner products of $C$. 

Designs with large $t$ and small $s$ are clearly interesting. 
Delsarte--Goethals--Seidel \cite{DGS} proved that $t \leq 2s$, 
and $t \leq 2s - 1$ if the set $A(C) \cup \{ 1 \}$ is symmetric 
with respect to 0 and discuss the cases of equality. On the other hand, 
it is shown in \cite{DGS} (see Theorem \ref{thm:DGS} below)
that $t \geq 2s - 2$ implies that $C$ 
carries an $s$-c1ass association scheme. 

Delsarte--Goethals--Seidel \cite{DGS} also proved the bound 
\begin{equation} \label{dgs-bound}
  |C| \geq {n+m-1-\varepsilon \choose n-1}+{n+m-2 \choose n-1}
\end{equation}
for any spherical $t$-design $C \subset \mathbb{S}^{n-1}$,
where $t=2m-\varepsilon$, $\varepsilon \in \{0,1\}$. 
A design is said to be {\it tight} if it attains \eqref{dgs-bound}.
Tight spherical designs were considered by Bannai--Damerell 
\cite{BD1,BD2} where it was proved that tight $2m$-designs
do not exist for $m \geq 3$ and tight $(2m-1)$-designs
do not exist for $m \geq 5$ except for the tight 11-design
defined by the Leech lattice in 24 dimensions. Further nonexistence 
results for tight 4-, 5-, and 7-designs were proved in \cite{BMV04,NV13}. 

Levenshtein \cite{Lev92} proved (in the more general setting of
{\it polynomial metric spaces}) that the codes with $t \geq 2s-1$ 
or even $t \geq 2s-2$ if the code is diametrical are maximal, that is
they attain what is now known as Levenshtein bound (see \cite{Lev-chapter}). 
Such codes were called {\it Delsarte codes} in polynomial metric spaces \cite{Lev92}.

Cohn--Kumar \cite{CK07} considered spherical codes with $t=2s-1$ or $t=2s-2$ which were called {\it sharp configurations} and appeared
to be {\it universally optimal} since they have 
the minimum possible energy for a large class of potential functions\footnote{More precisely, for all absolutely monotone potentials.}. 
Boyvalenkov--Dragnev--Hardin--Saff--Stoyanova \cite{BDHSS16} obtained an energy counterpart of the Levenshtein bound which is attained by all 
sharp configurations. 

In this paper we prove that the inner products of spherical $s$-distance $t$-designs with $t \geq 2s-2$ and $s \geq 3$ are rational with the only exception being the icosahedron. In other words, we prove that all sharp configurations have rational inner products and, still in other words, all spherical codes which attain the Levenshtein bound, have rational inner products, except for the icosahedron. 

The rationality problem was considered from the very beginning (see Theorem 7.7 in \cite{DGS}). Bannai--Damerell \cite{BD1,BD2} proved and applied the rationality of the inner products of tight spherical designs in order to prove the nonexistence results, mentioned above. The case $(t,s)=(3,2)$ was considered in \cite{BS04}. Note that there are spherical 2-distance 2-designs with irrational inner products, which are called conference graphs.
Rationality of inner products of antipodal $s$-distance sets of large cardinalities was proved in \cite{Noz}.  

The list of all known spherical $s$-distance $t$-designs with $t \geq 2s-2$ and $t \geq 3$ is unchanged since 1987 when Levenshtein \cite{Lev87} noticed that an infinite series of spherical $2$-distance 3-designs can be added to the examples from Delsarte--Goethals--Seidel \cite{DGS}. In particular, remarkable optimal codes are good kissing number configurations \cite{OS79}. 

The paper is organized as follows. In section 2 we use 
algebraic tools for proving the main result in the case $s \geq 6$ 
and derive an important corollary for the small cases. Section 3 is devoted to detailed investigation of the cases $s=3$, 4 and 5. In section 4 we present certain consequences and a different proof for the case $t=2s-1$
which works for $s \geq 2$.

\section{The case $s\geq 6$}

In this section we consider the cases $(s,t)=(s,2s-1)$ and $(s,2s-2)$ 
for $s\geq 6$. 
The spherical sets have the structures of  $Q$-polynomial association schemes \cite{DGS}. 
We will prove the rationality of inner products that appear in the spherical set by
using an automorphism of the Bose--Mesner algebra of the association scheme given from the Galois group of its splitting field  \cite{MW09} and Suzuki's result on multiple $Q$-polynomial structures \cite{S98}.  First, we give several definitions and known related results. 

Let $X$ be a finite set, and $\mathcal{R}=\{R_i\}_{i=0}^d$ be disjoint binary relations on $X$, where $R_0=\{(x,x) \mid x \in X\}$.  The pair $(X,\mathcal{R})$ is a (symmetric) {\it association scheme} of class $d$ if the following conditions hold:
\begin{enumerate}
    \item $X\times X=R_0\cup R_1 \cup \cdots \cup R_d$.
    \item For each $i\in \{0,1,\ldots, d\}$, if $(x,y) \in R_i$, then $(y,x) \in R_i$. 
    \item For any $i,j,k \in \{0,1,\ldots, d\}$, there exists an integer $p_{ij}^k$ such that for each $(x,y) \in R_k$ it follows    
    $p_{ij}^k=|\{z : (x,z) \in R_i, (z,y) \in R_j \}|$.
\end{enumerate}
Let $A_i$ be the symmetric matrix whose rows and columns are indexed by $X$ with $(x,y)$ entries
\[
(A_i)_{xy}=\begin{cases}
1 \text{ if $(x,y) \in R_i$},\\
0 \text{ otherwise}. 
\end{cases}
\]
The linear space spanned by $\{A_i\}_{i=0}^d$ over $\mathbb{C}$ is called the {\it Bose--Mesner algebra} of an association scheme, and it has two structures of commutative algebra with usual matrix multiplication and entrywise multiplication. There exist the primitive idempotents $\{E_i\}_{i=0}^d$ of the algebra with the usual multiplication. Namely, it satisfies that $E_i E_j=\delta_{ij} E_i$ with Kronecker's delta $\delta_{ij}$. The matrix $E_i$ is a positive semidefinite matrix with equal diagonal entries, and it is an orthogonal projection matrix onto a same eigenspace of $\{A_i\}_{i=0}^d$. 
An association scheme is {\it $Q$-polynomial} with respect to the ordering $E_0,E_1,\ldots, E_d$ (or with respect to $E_1$) if for each $i \in \{0,1,\ldots, d\}$, there exists a polynomial $v_i(x)$ of degree $i$ such that $E_i=v_i(E_1)$ where we use the entrywise multiplication. See \cite{BIb} for more details on association schemes. 

Suzuki \cite{S98} proved the following theorem on the multiple structures of $Q$-polynomial association schemes. 
\begin{theorem}\label{thm:suzuki}
Let $\mathcal{X}=(X,\{R_i\}_{i=0}^d)$ be a $Q$-polynomial association scheme with respect to the ordering $E_0,E_1,\ldots, E_d$ such that the rank of $E_1$ is greater than 2.  If $\mathcal{X}$ is $Q$-polynomial with respect to another ordering, then the new ordering is one of the following. 
\begin{enumerate}
    \item[(I)] $E_0,E_2,E_4,E_6,\ldots,E_5,E_3, E_1$.
    \item[(II)] $E_0,E_d,E_1,E_{d-1},E_2,E_{d-2},E_3,E_{d-3}, \ldots$
    \item[(III)] $E_0,E_d,E_2,E_{d-2},E_4,E_{d-4}, \ldots E_{d-5},E_5,E_{d-3},E_3,E_{d-1},E_1$.
    \item[(IV)] $E_0,E_{d-1},E_2,E_{d-3},E_4,E_{d-5}, \ldots, E_5,E_{d-4},E_3,E_{d-2},E_1,E_d$.
    \item[(V)] $d=5$ and $E_0,E_5,E_3,E_2,E_4,E_1$. 
\end{enumerate}
Moreover $\mathcal{X}$ has at most two $Q$-polynomial structures. 
\end{theorem}

The {\it splitting field} $\mathbb{F}$ of an association scheme is the smallest extension of the rationals $\mathbb{Q}$ containing the eigenvalues of all $A_i$ \cite{MW09,M91}. Indeed, the splitting field coincides with the smallest extension of $\mathbb{Q}$ containing the entries of all $E_i$. 
We consider the algebra 
\[ \mathfrak{A}'={\rm Span}_{\mathbb{F}}\{A_i\}_{i=0}^d={\rm Span}_{\mathbb{F}}\{E_i\}_{i=0}^d. \] 
Then a field automorphism $\sigma$ in ${ \rm Gal}( \mathbb{F}/\mathbb{Q} )$
induces the algebra automorphism of $\mathfrak{A}'$ (for the both multiplications) by entrywise action $(m_{ij})^\sigma:=(m_{ij}^\sigma)$. 
The field automorphism $\sigma$ faithfully acts on the primitive idempotents $\{E_i\}_{i=0}^d$ and $A_i^\sigma=A_i$ \cite{MW09}.  
If there exists an irrational entry in $E_i$, then
there exists $\sigma \in {\rm Gal }(\mathbb{F}/ \mathbb{Q})$ such that 
$E_i^\sigma \ne E_i$. For such $\sigma$, a $Q$-polynomial  ordering $E_0,E_1,\ldots, E_d$ becomes the other $Q$-polynomial ordering $E_0^\sigma=E_0, E_1^\sigma,\ldots, E_d^\sigma$ \cite{MW09}. 

Delsarte, Goethals, and Seidel \cite{DGS} showed that an $s$-distance set with high strength as spherical design has a connection to association scheme as follows.

\begin{theorem}[{Delsarte--Goethals--Seidel \cite{DGS}}] \label{thm:DGS}
Let $C$ be a spherical $s$-distance $t$-design. For inner products $a_i \in A(C)$ with $1=a_0> a_1 > \cdots >a_s$, we define the relations $R_i=\{(x,y) \in C \times C \mid \langle x,y \rangle =a_i \}$. 
If $t\geq 2s-2$ holds, then $(C,\{R_i\}_{i=0}^s)$ is 
an association scheme.  
\end{theorem}

For the association scheme obtained from Theorem~\ref{thm:DGS}, the primitive idempotents are written by characteristic matrices $H_k$. We introduce the relationship between spherical design and characteristic matrices. 
A polynomial with real coefficients in $n$ variables $\xi_1,\ldots,\xi_n$ is {\it harmonic} if it is in the kernel of the
Laplacian $\sum_{i=1}^n \partial^2/ \partial\xi_i^2$.
Let ${\rm Harm}(n,k)={\rm Harm}(k)$ be the linear space of the homogeneous harmonic polynomials of degree $k$ in $n$ variables. 
The dimension of ${\rm Harm}(n,k)$ is 
\[ h_{n,k}=h_k=\binom{n+k-1}{k}-\binom{n+k-3}{k-2}. \]
Let $\{W_{k,i}\}_{i=1}^{h_k}$ be an orthonormal basis of ${\rm Harm}(k)$ with respect to the inner product 
\[ \int_{\mathbb{S}^{n-1}} f(x) g(x) d\mu(x). \] 
For a finite set $C\subset \mathbb{S}^{d-1}$ and an orthonormal basis $\{W_{k,i}\}_{i=1}^{h_k}$, the {\it $k$-th characteristic matrix} $H_k$ is defined to be  the $|C| \times h_k$ matrix $H_k=(W_{k,i}(x))_{x\in C, i \in \{1,\ldots , h_k\}}$. 
\begin{theorem}[{\cite{DGS}}] \label{thm:ch_des}
Let $C$ be a finite subset of $\mathbb{S}^{n-1}$ and $H_k$ be a characteristic matrix of $C$ (we may take any basis of ${\rm Harm}(k)$).  
Then $C$ is a spherical $t$-design if and only if $H_k^\top H_l=|C| \Delta_{k,l}$ for $0\leq k+l \leq t$, where $\Delta_{k,l}$ is the identity matrix if $k=l$, the zero matrix otherwise. 
\end{theorem}
The following is an expression of the primitive idempotents of the association scheme obtained from Theorem~\ref{thm:DGS}. 
\begin{theorem}[{\cite{D73,DGS}}]
Let $C$ be an $s$-distance $t$-design with $t \geq 2s -2$, which has the structure of an association scheme. Then the primitive idempotent $E_i$ can be expressed by 
\[
E_i=\frac{1}{|C|} H_iH_i^\top
\]
for $i \in \{0,1,\ldots, s-1\}$, and $E_s=I-\sum_{i=0}^{s-1}E_i$, where $I$ is the identity matrix.  
\end{theorem}

\begin{remark} \label{rem:DGS}
For $t\geq 2s-2$, $H_i^\top H_i$ is the identity matrix of degree $h_i$ for $i\in \{0,1,\ldots, s-1\}$ by Theorem~\ref{thm:ch_des}. 
This implies that the rank of $E_i$ is $h_i$ and
the matrix $E_i$ can be expressed by $|C|E_i=(h_i Q_i(\langle x,y \rangle))_{x,y \in C}$ for $i\in \{0,1,\ldots, s-1\}$, where $Q_i$ is the Gegenbauer polynomial of degree $i$ normalized by $Q_i(1)=1$. 
From this fact, $E_1$ is identified with the Gram matrix of $C$, and the association scheme obtained from $C$ is a $Q$-polynomial association scheme with respect to $E_1$.
\end{remark}

Now we prove the rationality of the inner products. 

\begin{theorem} \label{thm:s>5}
Let $s$ be an integer greater than $5$, and $n$ an integer greater than $2$. 
Let $C$ be a spherical $s$-distance $t$-design in $\mathbb{S}^{n-1}$. 
If $t \geq 2s-2$ holds, then the inner product between any two points in $C$ is rational. 
\end{theorem}
\begin{proof}
With our assumption, $C$ has the structure of a $Q$-polynomial association scheme of class $s$ by Theorem \ref{thm:DGS} and Remark \ref{rem:DGS}. 
Let $\{E_i\}_{i=0}^s$ be the primitive idempotents of the Bose--Mesner algebra of the association scheme, where $E_1$ is identified with the Gram matrix of $C$. By Remark \ref{rem:DGS}, the rank of $E_k$ is equal to $h_{n,k}$ for each $k \in \{0,\ldots, s-1\}$. In particular, the rank of $E_1$ is different from $E_k$ except for $k=s$. 

Assume there exists an irrational inner product in $A(C)$, namely there exists an irrational entry $a$ in $E_1$. Let $\mathbb{F}$ be the splitting field of the association scheme. 
Then there exists a field automorphism $\sigma$ in ${\rm Gal}(\mathbb{F}/\mathbb{Q})$ which does not fix $a$. 
 The automorphism $\sigma$ faithfully acts on the primitive idempotents $\{E_i\}_{i=0}^s$, and does not fix $E_1$. Since the rank of $E_k^\sigma$ is the same as  $E_k$, we must have $E_1^\sigma=E_s$ and $E_k^\sigma=E_k$ for each $k \in \{2,\ldots, s-1\}$. This implies that the association scheme has the other $Q$-polynomial structure with ordering $E_0,E_s,E_2,E_3,\ldots,E_{s-1},E_1$. 
By Theorem \ref{thm:suzuki}, the possible cases are (II) with $s=2$, or (III) with $s=3,4,5$. Therefore, for $s\geq 6$, the inner products are all rationals. 
\end{proof}

\section{The cases $s=3, 4, 5$}

The small cases $s=3$, 4, 5 are dealt by careful investigation of the distance 
distributions of the corresponding codes.
Let $A(C)=\{a_1,a_2,\ldots,a_s\}$ be the nontrivial inner products in $C$ 
satisfying 
\[ -1 \leq a_s <a_{s-1} <\cdots<a_1<1. \]
For fixed $x \in C$ and $a \in A(C)$, let $A_a(x):=|\{ y \in C: \langle x,y \rangle =a \}|$.  Then the system of nonnegative integers
\[ \left( A_{a_1}(x),A_{a_2}(x),\ldots, A_{a_s}(x) \right) \]
is called the {\it distance distribution} of $C$ with respect to $x$. 

Let $C \subset \mathbb{S}^{n-1}$ be a spherical $s$-distance $t$-design for $n \geq 3$, $s \geq 3$, and $t \geq 2s-2$. 
Then the numbers $A_{a_i}(x)$ do not depend on $x$ (so we omit $x$ in the sequel) and satisfy the equations
\begin{equation} \label{syst1}
\sum_{i=1}^s a_i^j A_{a_i}=f_j |C|-1, \ \ j=0,1,\ldots,s, 
\end{equation}
where $f_j=0$ for odd $j$, $f_0=1$, $f_{2i}=(2i-1)!!/n(n+2)\cdots(n+2i-2)$
for $1 \leq i \leq [t/2]$ \cite{DGS}. We will use below the design properties (i.e., the equations from \eqref{syst1}) to analyze the possibilities for the distance distributions and the inner products. 

\subsection{The cases $s=4, 5$}

\begin{theorem} \label{s=4,5} 
Let $s=4$ or 5, and $n$ be an integer greater than $2$. Let $C$ be a spherical $s$-distance $t$-design in $\mathbb{S}^{n-1}$, where $t \geq 2s-2$. 
Then the inner product between any two points in $C$ is rational. 
\end{theorem}

\begin{proof}
It follows from the end of the proof of Theorem \ref{thm:s>5} (see also Remark 
\ref{rem:DGS}) that $Q_k(a)$ is rational for each inner product $a$ and $k=2,\ldots,s-1$. Using this for $k=2$, we conclude that all 
inner products are of the form $\pm \sqrt{b}$ for some rational $b$. 
Using the same fact for $k=3$, we see that $\pm \sqrt{b}(ub+v)$ is 
rational (with $u$ and $v$ rational), which means that either $\sqrt{b}$
is rational or $ub+v=0$. In the first case we are done, and in the second 
case the explicit form of the Gegenbauer third degree polynomial implies 
that $b=3/(n+2)$. 

Therefore, we may have only $\pm \sqrt{3/(n+2)}$ as possible irrational inner products. Clearly, both should appear with, moreover,  $A_{-\sqrt{3/(n+2)}}=A_{\sqrt{3/(n+2)}}$, following trivially from \eqref{syst1} for $j=1$. 
We consider separately the cases $s=4$ and $s=5$, where $C$ has $s-2=2$ or 3 further rational inner products. 

{\sl Case 1.} $s=4$. We denote the two rational inner products by $a$ and $b$,
where $a<b$, and the corresponding entries from the distance distribution by $X$ and $Y$. The design properties \eqref{syst1} for odd $i$ (note that $t \geq 2s-2 \geq 6$) imply that 
\[ Xa+Yb=Xa^3+Yb^3=Xa^5+Yb^5=-1. \] 
If $a^2\ne b^2$ and $ab \neq 0$, the first two equations give 
\[ X=-\frac{1-b^2}{a(a^2-b^2)}, \ Y=-\frac{1-a^2}{b(b^2-a^2)}. \] Then, $$-1=Xa^5+Yb^5=\frac{a^4(1-b^2)}{b^2-a^2}-\frac{b^4(1-a^2)}{b^2-a^2}=a^2b^2-a^2-b^2,$$
whence $(1-a^2)(1-b^2)=0$. This implies $a=-1$ and then $Y=0$, which is a contradiction. 

If $a^2=b^2$, it follows that $a=-b$. Then $-1=aX+bY=a(X-Y)$ and
$-1=a^3X+b^3Y=a^3(X-Y)$, implying $a=-1$, a contradiction with $a=-b$. 

If $ab=0$, let for example $b=0$. Then $Xa=Xa^3=-1$ gives $a=-1$ and $X=1$,
i.e. the inner products are $-1$, $0$ and $\pm \sqrt{3/(n+2)}$; note that 
the assumption $a=0$ would already contradict to $a<b$. Using \eqref{syst1} for $i=2$ and 4, we have  
\[ 2\cdot \frac{3}{n+2} \cdot A_{\sqrt{3/(n+2)}} =\frac{|C|}{n}-2 \] and 
\[ 2 \cdot \frac{9}{(n+2)^2} \cdot A_{\sqrt{3/(n+2)}}=\frac{3|C|}{n(n+2)}-2, \]
respectively, yielding the equality 
\[ \frac{3}{n+2}\cdot \left(\frac{|C|}{n}-2\right)=\frac{3|C|}{n(n+2)}-2, \]
which is only possible for $n=1$, a contradiction.

{\sl Case 2.} $s=5$. Denote by $a<b<c$ the three rational inner products and by
$X$, $Y$, and $Z$ the corresponding entries from the distance distribution. 
Since $t \geq 2s-2 \geq 8$, we have  
\[ Xa+Yb+Zc=Xa^3+Yb^3+Zc^3=Xa^5+Yb^5+Zb^5=Xa^7+Yb^7+Zc^7=-1 \]
from \eqref{syst1} for $i=1,3,5,7$, respectively. Assuming that there are no equal among $a^2$, $b^2$, and $c^2$ and $abc \neq 0$, we obtain 
\[ X=-\frac{(1-b^2)(1-c^2)}{a(a^2-b^2)(a^2-c^2)}, \ Y=-\frac{(1-a^2)(1-c^2)}{b(b^2-a^2)(b^2-c^2)}, \ Z=-\frac{(1-a^2)(1-b^2)}{c(c^2-b^2)(c^2-a^2)} \]
by the first three equations. Therefore, 
\[-1=Xa^7+Yb^7+Zc^7=-a^2-b^2-c^2+a^2b^2+a^2c^2+b^2c^2-a^2b^2c^2, \]
that is $(1-a^2)(1-b^2)(1-c^2)=0$. If $a=-1$, then $Y=Z=0$, 
a contradiction. 

Assume now (without loss of generality) that $b^2=c^2$ that is, $b=-c.$ This 
gives the equations
\[ Xa+(Y-Z)b=Xa^3+(Y-Z)b^3=Xa^5+(Y-Z)b^5=Xa^7+(Y-Z)b^7=-1, \]
which implies, in the same way as in the case $s=4$, that 
$a=-1$, $X=1$, and $Y-Z=0$ (note that still $abc \neq 0$). 
The equations for even $i$ from \eqref{syst1} give 
$$2\cdot \frac{3}{n+2} \cdot A_{\sqrt{3/(n+2)}}+2Yb^2 =\frac{|C|}{n}-2, $$ $$2 \cdot \frac{9}{(n+2)^2} \cdot A_{\sqrt{3/(n+2)}}+2Yb^4=\frac{3|C|}{n(n+2)}-2,$$
$$2 \cdot \frac{27}{(n+2)^3} \cdot A_{\sqrt{3/(n+2)}}+2Yb^6=\frac{15|C|}{n(n+2)(n+4)}-2.$$ $$2 \cdot \frac{81}{(n+2)^4} \cdot A_{\sqrt{3/(n+2)}}+2Yb^8=\frac{105|C|}{n(n+2)(n+4)(n+6)}-2,$$ 
respectively. Expressing $b^2$ from the first three equations yields  
\[ b^2=\frac{n(n+2)(n+4)-3|C|}{n(n+2)(n+4)} \]
and, similarly, the last three equations give 
\[ b^2=\frac{n(n+2)(n+4)(n+6)-30|C|}{(n+6)(n(n+2)(n+4)-3|C|)}. \]
This implies $3|C|=2n(n+1)(n+2)(n+4)/(n+6)$, whence in turn 
$b^2=\frac{4-n}{n+6}$, possible for $n=2$ and 3 only. Now 
it follows that $|C|=12$ and $n=2$, which gives $s\geq 6$. 

Finally, if $c=0$ (again without loss of generality), then  
\[ Xa+Yb=Xa^3+Yb^3=Xa^5+Yb^5=Xa^7+Yb^7=-1 \]
which is dealt in the same way as in the case $s=4$ with $a^2 \neq b^2$. 
\end{proof}

\subsection{The case $s=3$}

We will use the following fact which can be referred to as 
Besicovitch's theorem \cite{Bes}. If $n_1,n_2,\ldots,n_k$ are mutually distinct squarefree positive integers and $b_1,b_2,\ldots,b_k$ are rationals, then the equality 
\begin{equation} \label{bes-equ}
  b_1\sqrt{n_1}+b_2\sqrt{n_2}+\cdots+b_k\sqrt{n_k}=0
\end{equation}
is possible only when $b_1=b_2=\cdots=b_k=0$. Allowing equal $n_i$'s, 
it follows that \eqref{bes-equ} implies, possibly after some rearrangements, that $n_i=n_j$ and $b_i+b_j=0$ for some $1 \leq i<j \leq k$. 

\begin{theorem} \label{thm:3-5} If $C \subset \mathbb{S}^{n-1}$, $n \geq 3$, is a spherical $3$-distance $4$-design, 
then its inner products are rational or $n=3$ and $C$ is isometric to the icosahedron. 
\end{theorem}

\begin{proof} We recall that tight 5-designs could possibly exist only
in dimensions $n=m^2-2$ for some odd positive integer $m \geq 3$
or $n=3$ (the icosahedron) and their inner products are $-1$ and $\pm 1/m$ if $n=m^2-2$ or $-1$ and $\pm 1/\sqrt{5}$ if $n=3$. We will prove below
that $C$ is a 5-design, thus we may assume afterwards in the proof that $C$ is not tight. In particular, we will have $|C|>n^2+n$ by the Delsarte--Goethals--Seidel bound \eqref{dgs-bound}. 

As in the beginning of the proof of Theorem \ref{s=4,5} we conclude that
all irrational inner products of $C$ are of the form $\pm \sqrt{b}$ for some rational $b$. We could not use now the third Gegenbauer polynomial and proceed by a direct argument. 
Let the inner products of $C$ be $a_1$, $a_2$, and $a_3$. 
Combining the Besicovitch's theorem (or proceeding directly) 
and the 1-degree design property
\[ a_1A_{a_1}+a_2A_{a_2}+a_3A_{a_3}=-1 \]
we conclude that that, without loss of generality, $a_1=-a_2=\sqrt{b}$ and
$A_{a_1}=A_{a_2}$. Therefore $a_3A_{a_3}=-1$. 
The 3-degree design property gives $a_3^2=1$, i.e. $a_3=-1$
and $A_{-1}=1$ (i.e., $C$ is antipodal, therefore a 5-design). 
Further, we compute
$A_{a_1}=A_{a_2}=(|C|-2)/2$ and
\[ b=\frac{|C|-2n}{n(|C|-2)},\ b^2=\frac{3|C|-2n(n+2)}{n(n+2)(|C|-2)} \] 
by the 2- and 4-degree design conditions, respectively. Hence
\[ \left(\frac{|C|-2n}{n(|C|-2)}\right)^2=\frac{3|C|-2n(n+2)}{n(n+2)(|C|-2)},\]
yielding $(n-1)(|C|-n^2-n)=0$, i.e. $|C|=n^2+n$ and $C$ is a tight spherical 5-design, which is a contradiction. \end{proof}

\begin{remark} The splitting field of a $Q$-polynomial association scheme is an extension of rationals of degree at most 2 \cite[Theorem 2.2]{MW09}.
We may use this instead of the argument via Theorem \ref{thm:s>5} and the 
Besicovitch theorem. 
\end{remark}

\section{Consequences and reformulations}

The next theorem is implicit in \cite{Lev92} and a short proof can be found
in \cite{BS20}. 

\begin{theorem} \label{max-l-general} 
If $C \subset \mathbb{S}^{n-1}$ is a spherical $s$-distance $(2s-1)$-design,
which is not tight, then its inner products are exactly the roots 
of the Levenshtein polynomial
\begin{equation} \label{lev-equ}
P_{s}(u)P_{s-1}(r) - P_{s}(r)P_{s-1}(u)=0, 
\end{equation}
where $P_i(u)=P_i^{(\frac{n-1}{2},\frac{n-3}{2})}(u)$ is a Jacobi polynomial normalized for $P_i(1)=1$ and $r$ is
determined as the maximal root of the equation
\[ |C|=L_{2s-1}(n,u), \]
$L_{2s-1}(n,u)$ is the Levenshtein bound.
\end{theorem}

Combining Theorem \ref{max-l-general} with the fact that the inner products are $\pm \sqrt{b}$ and the Besicovitch's theorem implies that either all inner products are rational or there are two products which sum up to 0. The last gives a contradiction the following lemma from \cite{BDL99}.

\begin{lemma}[\cite{BDL99}]
Let $t=2s-1 \geq 3$ in the non-tight design case. If $a_0,a_1,\ldots,a_{s-1}$ 
are the roots of \eqref{lev-equ},
then $a_i+a_j \neq 0$ for every $i,j \in \{0,1,\ldots,s-1\}$.
\end{lemma}

This gives an alternative proof of the main result in the case $t=2s-1 \geq 3$. 

We present some different formulations. Our rationality result 
implies the following characterization of the spherical codes which attain the Levenshtein bounds.  

\begin{corollary} \label{lev-b-attaining}
All codes that attain Levenshtein bounds $L_t(n,u)$ for $t \geq 3$, 
apart from the icosahedron, have rational inner products.
\end{corollary} 

\begin{remark} We may have an alternative proof here as well. Indeed, 
there exist no codes attaining the even bounds $L_{2s}(n,u)$ for $s \geq 2$
apart from the tight 4-designs \cite{BDL99}, while 
for odd $t=2s-1$ the claim follows from the above discussion. 
\end{remark}

Cohn--Kumar \cite{CK07} called {\it sharp configurations} all spherical
$s$-distance $(2s-1)$-designs. Therefore we may reformulate 
Corollary \ref{lev-b-attaining} as follows. 

\begin{corollary}
All sharp configurations apart from the icosahedron have rational inner products. 
\end{corollary}

\begin{remark} \label{rem:bannai}
After we finished the first version of this manuscript,  
Bannai \cite{Bpc} suggested to write a simple proof of the rationality for $t=2s-1$ using only the parameters of the corresponding $Q$-polynomial association schemes. 
His idea is the use of the equality $m_s=m_{s-1} b_{s-1}^*/ c_{s}^*$ \cite[Ch.\ II.\ Proposition 3.7 (vi)]{BIb}. 
Under our assumption, we can obtain $m_{s-1}=h_{s-1}$, 
$b_{s-1}^*=n-n(s-1)/(n+2(s-1)-2)$ \cite[Theorem 3.1]{S11}, and  $c_{s}^* \leq n$, where $n$ is the dimension of the sphere $\mathbb{S}^{n-1}$. 
If the set $X\subset \mathbb{S}^{n-1}$ has an irrational inner product, then $m_s=n$ holds, which implies
\[
n=m_s=m_{s-1} \frac{b_{s-1}^*}{ c_{s}^*} \geq h_{s-1}\left(1-\frac{s-1}{n+2(s-1)-2} \right).
\]
However, this inequality is not satisfied for $s=3$, $n>3$, or $s=4,5$, $n \geq 3$. 
This method might be not applicable for $t=2s-2$ since $b_{s-1}^*$ depends on $a_{s-1}^*$ which is not determined. 
\end{remark}

\bigskip

\noindent
\textbf{Acknowledgments.} 
The authors thank Eiichi Bannai for providing the method to prove the rationality of inner products for $t=2s-1$ 
as explained in Remark~\ref{rem:bannai}.

\end{document}